\documentclass[12pt]{article}

\usepackage[margin=1.5in]{geometry}

\usepackage{amsmath}
\usepackage{amsthm}
\usepackage{amsfonts}
\usepackage{multicol}
\usepackage[numbers,sort&compress]{natbib}
\usepackage{booktabs}
\usepackage{titlesec}

\usepackage{enumitem}
\setlist[enumerate]{label={\upshape(\alph*)}}

\usepackage{color}

\usepackage{tikz}
\usetikzlibrary{calc}
\tikzstyle{vertex}=[circle, draw, inner sep=0pt, minimum size=4pt,fill=black]

\tikzstyle{hollowvertex}=[circle, draw, inner sep=0pt, minimum size=4pt]

\tikzstyle{namedvertex}=[circle, draw, inner sep=1pt, minimum size=12pt]

\tikzstyle{phantomvertex}=[circle, draw, inner sep=0pt, minimum size=4pt,color=white]

\usetikzlibrary{decorations.pathreplacing}
\newtheorem{theorem}{Theorem}[section]
\newtheorem{lemma}[theorem]{Lemma}

\newcommand{\real}{\mathbb{R}}
\newtheorem{proposition}[theorem]{Proposition}

\newtheorem{corollary}[theorem]{Corollary}

\theoremstyle{definition}

\newtheorem{question}{Question}

\begin{document}

\title{A note on purely imaginary independence roots
}
\author{Jason I. Brown\\
\small Department of Mathematics \& Statistics\\
\small Dalhousie University\\
\small jason.brown@dal.ca\\
\\
Ben Cameron\\
\small School of Computer Science\\
\small University of Guelph\\
\small ben.cameron@uoguelph.ca
}
\date{\today}

\maketitle

\begin{abstract}
The independence polynomial of a graph is the generating polynomial for the number of independent sets of each cardinality and its roots are called independence roots. We investigate here purely imaginary independence roots. We show that for all $k\ge 4$, there are connected graphs with independence number $k$ and purely imaginary independence roots. We also show that every graph is an induced subgraph of a connected graph with purely imaginary independence roots and classify every purely imaginary number of the form $ri$, $r$ rational, that is an independence root.

\vspace{0.25in}


\vspace{0.25in}

\end{abstract}

\section{Introduction}
For a (finite and simple) graph $G$ the \textit{independence polynomial} of $G$ is defined by $$I(G,x)=\sum_{k=0}^{\alpha(G)}s_kx^k$$
where $s_k$ is the number of independent sets of $G$ with cardinality $k$ and $\alpha(G)$ is the \textit{independence number} of $G$, that is, the cardinality of the largest independent set in $G$. The roots of $I(G,x)$ are called the \textit{independence roots} of $G$, and have been well studied, with results including:
\begin{itemize}
\item An independence root of smallest modulus is always real \cite{BDN2000,Csikvari2013}.
\item The roots of independence polynomials of claw-free graphs (that is, those that do not contain an induced $K_{1,3}$) are always real \cite{Chudnovsky2007}.
\item The closure of the real independence roots is $(-\infty,0]$, while the closure of all of the (complex) independence roots is the entire complex plane \cite{INDROOTS}.
\item Various bounds on the maximum modulus are known both for all graphs and trees of order $n$ (that is, on $n$ vertices) \cite{BrownCameron2018maxmod} and for graphs of order $n$ and fixed independence number \cite{BrownNowakowski2001}.
\end{itemize}

Much of the work for independence roots has followed the path of other graphs polynomials, and in particular \textit{chromatic polynomials} (the chromatic polynomial $\pi(G,x)$ of a graph $G$ counts the number of ways to properly colour the vertices with $x$ colours, when $x$ is a nonnegative integer). One of the outstanding questions about chromatic polynomials is whether they can have any nonzero purely imaginary roots, and Bohn  \cite{Bohn2014} states that ``it is widely suspected there are no purely imaginary chromatic roots.'' (For a complex number $z$, $z$ is \textit{purely imaginary} if $\operatorname{Re}(z)=0$, and if further $\operatorname{Im}(z)$ is rational, we say that $z$ is \textit{rational purely imaginary}.)

There have been no known examples of purely imaginary independence roots in the literature in spite of the fact that there are independence roots arbitrarily close to every purely imaginary number \cite{INDROOTS}. In this note we show that while no graph of small independence number (less than $4$) has purely imaginary independence roots, there are infinitely many connected graphs with purely imaginary independence roots and every graph is a subgraph of such a graph. We also show that connected graphs with purely imaginary independence roots exist for any given independence number greater than $3$ and classify which rational purely imaginary independence roots can exist.

We remark that a graph having a purely imaginary independence root at $i$ says something structurally interesting about the graph. By substituting in $x = i$ and collecting up the real and imaginary parts, we observe that a graph $G$ has an independence root at $i$ if and only if $G$ has the same number of independent sets with cardinality $0\ (\operatorname{mod } 4)$ as independent sets with cardinality $2\ (\operatorname{mod } 4)$ and the same number of independent sets with cardinality $1\ (\operatorname{mod } 4)$ as independent sets with cardinality $3\ (\operatorname{mod } 4)$.

\section{Purely Imaginary Independence Roots}

A simple observation is that  if $p(x) = \sum a_{i}x^{i}$ is a polynomial with real coefficients and $b\in\real$ ($b\neq0$), then $p(x)$ has a root at $bi$ if and only if $p_{even}(x) = \sum a_{2i}x^{i}$ and $p_{odd}(x) = \sum a_{2i+1}x^{i}$ both have roots at $-b^2$. From this, and the proof of \cite[Proposition 2.1]{BrownCameron2018stability} we have the following result.

\begin{proposition}\label{prop:smallindnum}
If $\alpha(G)\le 3$, then $G$ has no purely imaginary independence roots.
\end{proposition}
\begin{proof}
While the statement of \cite[Proposition 2.1]{BrownCameron2018stability} does not appear to include purely imaginary independence roots, the authors actually showed that if $\alpha(G)\le 2$, then $G$ has all real independence roots and if $\alpha(G)=3$, then $r<s$ where $r$ is the root of $I_{odd}(G,x)$ and $s$ is the root of $I_{even}(G,x)$. Given the above comment, this shows that for $\alpha(G)\le 3$, $G$ has no purely imaginary independence roots.
\end{proof}

To provide examples of graphs with purely imaginary independence roots, the following graph operations will be essential. For a graph $G$, $\overline{G}$ denotes the complement of $G$. The \textit{disjoint union} $G\cup H$ is the graph with vertex set $V(G)\cup V(H)$ and edge set $E(G)\cup E(H)$. The \textit{join} $G + H$ of two graphs $G$ and $H$ is formed from their disjoint union by adding in all edges between a vertex of $G$ and a vertex of $H$.

\subsection{Graphs via Joins}

We begin by showing that graphs with independence number $4$ and purely imaginary independence roots not only exist, but that there are infinitely many such graphs. To do this, we require the following technical lemma.

\begin{lemma}\label{lem:diophantine}
Let $x_1=y_1=1$, $x_2=5$, $y_2=3$ and for $n\ge 3$, $x_n=4x_{n-1}-x_{n-2}$ and $y_n=4y_{n-1}-y_{n-2}$. Then
\begin{itemize}
\item[i)] $x_n^2-3y_n^2=-2$ for all $n\ge 1$ and
\item[ii)] $x_{n-1}x_{n-2}-3y_{n-1}y_{n-2}=-4$ for all $n\ge 3$.
\end{itemize} 

\end{lemma}
\begin{proof}
We prove both statements simultaneously by induction on $n$. Both statements are readily verified for $n\le 4$ so suppose they hold for all $4\le k<n$. Now,

\begin{align*}
x_{n-1}x_{n-2}-3y_{n-1}y_{n-2}&=(4x_{n-2}-x_{n-3})x_{n-2}-3(4y_{n-2}-y_{n-3})y_{n-2}\\
&=4(x_{n-2}^2-3y_{n-2}^2)-(x_{n-3}x_{n-2}-3y_{n-3}y_{n-2})\\
&=4(-2)-(-4)\ \ \ \ \ \ \ \ \ \ \ \text{(by the inductive hypothesis)}\\
&=-4
\end{align*}
and
\begin{align*}
x_n^2-3y_n^2&=(4x_{n-1}-x_{n-2})^2-3(4y_{n-1}-y_{n-2})^2\\
&=16(x_{n-1}^2-3y_{n-1}^2)+x_{n-2}^2-3y_{n-2}^2-8(x_{n-1}x_{n-2}-3y_{n-1}y_{n-2})\\
&=-34-8(x_{n-1}x_{n-2}-3y_{n-1}y_{n-2})\ \ \ \ \text{(by the inductive hypothesis)}\\
&=-34-8(-4)\ \ \ \ \ \ \ \ \ \ \ \ \ \ \text{(by the inductive hypothesis and above)}\\
&=-2.
\end{align*} 
\end{proof}

For positive integers $a,b,c,d$, define the graph $G(a,b,c,d)$ by $G(a,b,c,d)=4K_{a}+3K_{b}+2K_{c}+K_{d}$ where $mK_{n}=\bigcup_{i=1}^{m}K_n$, i.e. the disjoint union of $m$ copies of $K_n$. Note that $G(a,b,c,d)$ is a special case of the graphs used by Alavi et al.~\cite{Alavi} for finding graphs with non-unimodal independence polynomials. We see that 
$$I(G(a,b,c,d),x)=(1+ax)^4+(1+bx)^3+(1+cx)^2+dx-2.$$

We note that since all independence polynomials have real coefficients, all nonreal roots come in complex conjugate pairs, so that a graph has $bi$ as an independence root if and only if $-bi$ is also an independence root. Thus through this note, we only need to argue that $bi$ is an independence root to conclude both $bi$ and $-bi$ are independence roots.
\begin{theorem}\label{thm:infmanyindnum4}
For $n\ge 1$, $G(a,b,c,d)$ has independence roots at $i$ and $-i$ if $a=3y_n$, $b=3y_nx_n$, $c=1$, and $d=4a^3-4a+b^3-3b-2$.
\end{theorem}
\begin{proof}
Suppose $a,b,c,d$ are as in the statement of the theorem. It is clear that $a,b,$ and $c$ are positive integers. Since $a^3-a\ge 6$ for all integers $a\ge 2$ and $b^3-3b\ge 2$ for all integers $b\ge 2$, it follows that $d$ is also a positive integer. Therefore, $G(a,b,c,d)$ is a well-defined graph and we may evaluate its independence polynomial at $i$ to obtain
\begin{align*}
I(G(a,b,c,d),x)&=(1+ai)^4+(1+bi)^3+(1+ci)^2+di-2\\
&=a^4-6a^2-3b^2-c^2+1+i(-4a^3+4a-b^3+3b+2c+d).
\end{align*} 

From this we find that $i$ is an independence root of $G(a,b,c,d)$ if and only if the following two equations hold:
\begin{align}
a^4-6a^2-3b^2&=0 \label{eq:realpart}\\
-4a^3+4a-b^3+3b+2+d&=0\label{eq:impart}
\end{align}

By the definition of $d$, it is clear that (\ref{eq:impart}) holds, so it remains only to show that (\ref{eq:realpart}) holds.

By Lemma~\ref{lem:diophantine} and since $a=3y_n$ and $b=3y_nx_n$, we have 
\begin{align*}
b&=3y_n\sqrt{3y_n^2-2}\\
&=\sqrt{27y_n^4-18y_n^2}\\
&=\sqrt{\frac{81y_n^4-54y_n^2}{3}}\\
&=\sqrt{\frac{(3y_n)^4-6(3y_n)^2}{3}}\\
&=\sqrt{\frac{a^4-6a^2}{3}}.
\end{align*}
Thus, equation (\ref{eq:realpart}) is satisfied.
\end{proof}

Note that $\{y_n\}_{n\ge 1}$ is an increasing sequence and $\alpha(G(a,b,c,d))=4$ regardless of the values of $a,\ b,\ c,$ and $d$, so we have found infinitely many connected graphs with independence number $4$ and purely imaginary independence roots. As the independence polynomial is multiplicative over disjoin union, trivially there are disconnected graphs of all independence numbers at least $4$ with $i$ and $-i$ as independence roots (and hence have purely imaginary independence roots). The problem becomes more interesting if we want to insist that the graphs are connected, and we spend the rest of this section building such graphs.

\begin{lemma}\label{lem:seedplus8k}
If $G$ is a graph with independence number $\alpha$ such that $G+K_d$ has an independence root at $i$ for some $d\ge 1$, then $(G\cup \overline{K_{8k}})+K_{16^kd}$ is a connected graph with independence number $\alpha+8k$ and independence roots at $i$ and $-i$ for all $k\ge 0$.
\end{lemma}
\begin{proof}
From the definitions of join and disjoint union, it is clear that $(G\cup \overline{K_{8k}})+K_{16^kd}$ is a connected graph with independence number $\alpha+8k$. Since $G+K_d$ has an independence root at $i$, it follows that $I(G,i)=-di$. Since $(1+i)^{8k}=16^k$, it follows that $I(G\cup \overline{K_{8k}},i)=-16^kdi$. Therefore, $I((G\cup \overline{K_{8k}})+K_{16^kd},i)=-16^kdi+16^kdi=0$. Thus, $(G\cup \overline{K_{8k}})+K_{16^kd}$ is a connected graph with independence number $\alpha+8k$ and an independence roots at $i$ and $-i$. 
\end{proof}

\begin{theorem}\label{thm:everyindnumber}
For all $\alpha\ge 4$, there exists a connected graph with independence number $\alpha$ and independence roots at $i$ and $-i$.
\end{theorem}
\begin{proof}
The proof involves giving a graph for each of $\alpha=4,5,6,7,8,9,10,11$ that will serve as a ``seed'' for each case of $\alpha\mod 8$. The result will then follow from Lemma~\ref{lem:seedplus8k}. The seed for $\alpha=4$ is given by any of the examples afforded by Theorem~\ref{thm:infmanyindnum4} by removing the $K_d$ joined to the graph. The seeds for $\alpha\ge 5$ are given in Table~\ref{tab:seeds} to complete the proof. Note that the graph $K_{n(k)}$ denotes the complete $n$-partite graph where each part has $k$ vertices, so that $\alpha(K_{n(k)})=k$.

\begin{table}[!h]
\begin{center}
\renewcommand\arraystretch{1.2}
\begin{tabular}{|c|c|c|c|}
\hline
$\alpha$ & $G$                        & $I(G,x)$          & $d$  \\ \hline
5        & $K_2\cup K_3\cup\overline{K_3}$  &  $(1+2x)(1+3x)(1+x)^3$  & 20   \\ \hline
6        & $\overline{K_6}$               &  $(1+x)^6$    & 8    \\ \hline
7        & $K_{3,5,7}$                     & $(1+x)^3+(1+x)^5+(1+x)^7-2$    & 10   \\ \hline
8        & $\overline{K_8}+K_{16(6)}$ & $(1+x)^8+16(1+x)^6-16$  &  128 \\ \hline
9        &   $\overline{K_9}+K_{16(6)}$ & $(1+x)^9+16(1+x)^6-16$  & 112    \\ \hline
10       & $4K_2\cup 4K_3\cup\overline{K_2}$ & $(1+2x)^4(1+3x)^4(1+x)^2$  & 5000 \\ \hline
11       & $3K_2\cup 3K_3\cup\overline{K_5}$  &  $(1+2x)^3(1+3x)^3(1+x)^5$ & 2000 \\ \hline
\end{tabular}
\caption{Seed graphs $G$ for $5\le\alpha(G)\le 11$ and associated value of $d$.}\label{tab:seeds}
\end{center}
\end{table}

\end{proof}

\subsection{Graphs via Lexicographic Products}

Thus far, the only purely imaginary independence roots we have found have been $i$ and $-i$. To find other purely imaginary numbers as independence roots we need to turn to another graph operation. Given graphs $G$ and $H$ such that $V(G)=\{ v_1,v_2,...,v_n\}$ and $V(H)=\{ u_1,u_2,...,u_k\}$, the \textit{lexicographic product} (or graph substitution) of $G$ and $H$, which we will denote $G[H]$, is the graph such that $V(G[H] )=V(G)\times V(H)$ and $(v_i,u_l)\sim (v_j,u_m)$ if $v_i\sim _G v_j$ or $i=j$ and $u_l\sim _H u_m$. See Figure~\ref{fig:lexprod} for an example.  The graph $G[H]$, can be thought of as substituting a copy of $H$ for each vertex of $G$. 

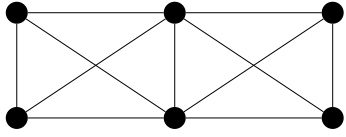
\begin{figure}[htp]
\def\c{0.7}
\def\r{1}
\centering
\scalebox{\c}{
\begin{tikzpicture}
\begin{scope}[every node/.style={circle,thick,draw,fill}]
    \node (1) at (-3*\r,0*\r) {};
    \node (2) at (0*\r,0*\r) {};
    \node (3) at (3*\r,0*\r) {};
    
    \node (4) at (-3*\r,-2*\r) {};
    \node (5) at (0*\r,-2*\r) {};
    \node (6) at (3*\r,-2*\r) {};
 
\end{scope}

\begin{scope}
    \path [-] (1) edge node {} (2);
    \path [-] (2) edge node {} (3);
    
    \path [-] (1) edge node {} (4);
    \path [-] (1) edge node {} (5);
    
    \path [-] (2) edge node {} (4);
    \path [-] (2) edge node {} (5);
    \path [-] (2) edge node {} (6);
    
    \path [-] (3) edge node {} (5);
    \path [-] (3) edge node {} (6);
    
    \path [-] (4) edge node {} (5);
    \path [-] (6) edge node {} (5);
  
\end{scope}

\end{tikzpicture}}
\caption{The lexicographic product $P_3[K_2]$.}\label{fig:lexprod}
\end{figure}

\begin{theorem}[\cite{INDROOTS}]\label{thm:subsformula}
For all graphs $G$ and $H$, $I(G[H],x)=I(G,I(H,x)-1)$.
\end{theorem}

It follows that if $z_1,z_2,\ldots,z_k$ are the independence roots of $G$, then the independence roots of $G[H]$ are all solutions to the equations $i(H,x)-1=z_i$ for all $i=1,2,\ldots,k$. If we let $H=K_n$, then the independence roots of $G[H]$ are $\frac{z_1}{n},\frac{z_2}{n},\ldots,\frac{z_k}{n}$. This leads to the following.

\begin{lemma}\label{lem:itoiovern}
Let $n$ be a positive integer. If $G$ has independence roots at $i$ and $-i$, then $G[K_n]$ has independence roots at $\frac{i}{n}$ and $-\frac{i}{n}$.
\end{lemma}

\begin{theorem}\label{thm:infmanyiovernroots}
For all $n\in\mathbb{Z}\setminus\{0\}$ and $\alpha\ge 4$, there exists a connected graph with independence number $\alpha$ and independence roots at $\frac{i}{n}$ and $-\frac{i}{n}$.  
\end{theorem}
\begin{proof}
The proof follows by taking a connected graph with the desired independence number and independence roots at $i$ and $-i$ from Theorem~\ref{thm:everyindnumber} and taking the lexicographic product of them with $K_n$, so that from Lemma~\ref{lem:itoiovern}, the resulting graph has independence roots at $\frac{i}{n}$ and $-\frac{i}{n}$. Since $G[K_n]$ preserves the independence number and connectivity of $G$, the resulting graph has the desired properties.
\end{proof}


Now that we have found infinitely many rational purely imaginary numbers that are independence roots, it is natural to wonder if we have found them all. The next result ensures that Theorem~\ref{thm:infmanyiovernroots} does indeed locate all rational purely imaginary independence roots.

\begin{proposition}\label{prop:rationalpureimroots}
Let $b\in\mathbb{Q}$. Then $bi$ is an independence root of some graph if and only if $b=\frac{1}{n}$ for some nonzero integer $n$. Moreover, if $|b|>1$, then neither $bi$ nor $\sqrt{b}i$ is not an independence root of any graph. 
\end{proposition}
\begin{proof}
The existence of rational purely imaginary roots of the form $\tfrac{i}{n}$ follows from Theorem~\ref{thm:infmanyiovernroots}. If $bi$ is an independence root of $G$ for a rational number $b$, then $-b^2$ is a rational root of $I_{even}(G,x)$. By the Rational Roots Theorem, any rational root must have numerator which divides the constant term of $I_{even}(G,x)$, which is $1$ for every graph $G$. Therefore, $-b^2=\frac{-1}{m}$ for some integer positive integer $m$. Therefore, $b$ must equal $\frac{1}{n}$ for some nonzero integer $n$. This argument also shows that if $|b|>1$, then $bi$ is not an independence root of any graph. 

Finally, if $|b|>1$ and $\sqrt{b}i$ is an independence root, then $-b$ must be a root of the even part of the corresponding independence polynomial. Therefore, the numerator of $b$ must divide $1$, so $|b|\le 1$, a contradiction.
\end{proof}

\subsection{Graphs via Coronas}

%
%
The \textit{corona} of a graph $G$ with a graph $H$, denoted $G\circ H$, is obtained by starting with the graph $G$, and for each vertex $v$ of $G$, joining a new copy of $H$ to $v$. For example, the star $K_{1,n}$ can be thought of as $K_1\circ \overline{K_n}$.  See Figure~\ref{fig:corona} for another example. In this section, we will use the corona to show that every graph is an induced subgraph of a connected graph with purely imaginary independence roots.

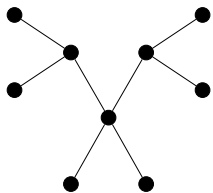
\begin{figure}[htp]
\def\c{0.5}
\def\r{1}
\centering
\scalebox{\c}{
\begin{tikzpicture}
\begin{scope}[every node/.style={circle,thick,draw,fill}]
    \node (1) at (-1*\r,0*\r) {};
    \node (2) at (1*\r,0*\r) {};
    \node (3) at (0*\r,-1.73205*\r) {};
    
    \node (4) at (2.5*\r,1*\r) {};
    \node (5) at (2.5*\r,-1*\r) {};
    
    \node (6) at (1*\r,-3.5*\r) {};
    \node(7) at (-1*\r,-3.5*\r) {};
    
    \node (8) at (-2.5*\r,1*\r) {};  
    \node (9) at (-2.5*\r,-1*\r) {};  
\end{scope}

\begin{scope}
    \path [-] (1) edge node {} (3);
    \path [-] (2) edge node {} (3);
    
    \path [-] (2) edge node {} (4);
    \path [-] (2) edge node {} (5);
    
    \path [-] (1) edge node {} (8);
    \path [-] (1) edge node {} (9);
    
    \path [-] (3) edge node {} (6);
    \path [-] (3) edge node {} (7);
\end{scope}

\end{tikzpicture}}
\caption{The graph $P_3\circ\overline{K_2}$}\label{fig:corona}
\end{figure}

\begin{theorem}[\cite{Gutman1992}]\label{thm:coronaformula}
If $G$ and $H$ are graphs and $G$ has order $n$, then $$I(G\circ H,x)=I\left(G,\tfrac{x}{i(H,x)}\right)i(H,x)^n.$$
\end{theorem}


\begin{proposition}\label{prop:3mod4}
Let $G$ be a graph of order $n$ with $n\equiv 3\ (\textnormal{mod}\ 4)$ and let $m$ be the integer $2^nI(G,\frac{1}{2})$. Then $(G\circ\overline{K_2})+K_m$ has independence roots at $i$ and $-i$.
\end{proposition}
\begin{proof}
Let $n\equiv 3\ (\textnormal{mod}\ 4)$, $G$ be a graph of order $n$, and $m=2^nI(G,\frac{1}{2})$. We now have,

\begin{align*}
I(G\circ\overline{K_2},i)&=I\left(G,\frac{i}{(1+i)^2}\right)(1+i)^{2n}\ \ \ \ \text{(by Theorem~\ref{thm:coronaformula})}\\
&=\sum_{k=0}^{\alpha(G)}s_ki^k(1+i)^{2(n-k)}\\
&=\sum_{k=0}^{\alpha(G)}s_k(-2^{n-k}i) \ \ \ \ \ \ \ \ \ \ \ \ \ \ \ \ \ \ \ \ \ \ \ \ \text{(by Table~\ref{tab:negimvalues}})\\
&=-2^ni\cdot I(G,\tfrac{1}{2})\\
&=-mi.
\end{align*}                                                                                                        

Therefore, 
$$I((G\circ\overline{K_2})+K_m,i)=I(G\circ\overline{K_2},i)+mi=0.$$
Hence, $(G\circ\overline{K_2})+K_m$ has independence roots at $i$ and $-i$.
\end{proof}

\begin{table}[h]
\begin{center}
\begin{tabular}{|c|c|c|c|c|}
\hline
$k\ (\operatorname{mod } 4)$ & $i^k$ & $2(n-k)\ (\operatorname{mod } 4)$ & $(i+1)^{2(n-k)}$ & $i^k(i+1)^{2(n-k)}$ \\ \hline
0                            & 1     & 3                                 & $-2^{n-k}i$      & $-2^{n-k}i$        \\ \hline
1                            & $i$   & 2                                 & $-2^{n-k}$          & $-2^{n-k}i$        \\ \hline
2                            & $-1$  & 1                                 & $2^{n-k}i$     &  $-2^{n-k}i$       \\ \hline
3                            & $-i$  & 0                                 & $2^{n-k}$           & $-2^{n-k}i$        \\ \hline
\end{tabular}
\caption{Values for $i^k(i+1)^{2(n-k)}$ for $n\equiv 3\ (\operatorname{mod }4)$.}\label{tab:negimvalues}
\end{center}   
\end{table}

\begin{theorem}\label{thm:everygraphsubgraph}
Let $k\in \mathbb{Z}$ be nonzero. Then every graph is an induced subgraph of a connected graph with $\frac{i}{k}$ and  $-\frac{i}{k}$ as independence roots.
\end{theorem}
\begin{proof}
Let $G$ be a graph or order $n$ with $n\equiv \ell\ (\operatorname{mod } 4)$, $0\le \ell \le 3$. Let $H$ be a graph of order $n+(3-\ell)$ such that $G$ is an induced subgraph of $H$ (note that many exist, for instance $H=G+K_{3-\ell}$). Now $n+(3-\ell)\equiv 3\ (\operatorname{mod } 4)$ so from Proposition~\ref{prop:3mod4}, $(H\circ \overline{K_2})+K_m$ has independence roots at $i$ and $-i$ for some $m\ge 1$. Therefore $G$ is an induced subgraph of a graph with independence roots at $i$ and $-i$. Now from Lemma~\ref{lem:itoiovern}, $\left((H\circ \overline{K_2})+K_m\right)[K_k]$ has independence roots at $\frac{i}{k}$ and $-\frac{i}{k}$. Finally, the subset of vertices formed by taking one vertex from each of the cliques substituted into each vertex of the induced copy of $G$ in $(H\circ \overline{K_2})$ induces a copy of $G$ in  $\left((H\circ \overline{K_2})+K_m\right)[K_k]$. 
\end{proof} 

\section{Conclusion}
We remark that the graphs found in Theorem~\ref{thm:everygraphsubgraph} can be extremely large. For example, if $G$ is the tree in Figure~\ref{fig:treewithlargeclique} of order $n=14$, then the smallest connected graph from our construction that contains $G$ as an induced subgraph with purely imaginary independence roots will have order $1509397$.

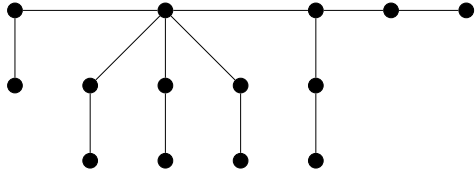
\begin{figure}[htb]
\centering
\def \n {8}
\def \r {2}
\def \radius {3cm}
\def \margin {8} 
\scalebox{0.5}{
\begin{tikzpicture}
\begin{scope}[every node/.style={circle,thick,draw,fill}]
    \node (14) at (0*\r,3*\r) {};
    \node (2) at (-1*\r,2*\r) {};
    \node (3) at (0*\r,2*\r) {};
    \node (4) at (1*\r,2*\r) {};
    \node (9) at (-1*\r,1*\r) {};
    \node (10) at (0*\r,1*\r) {};
    \node (11) at (1*\r,1*\r) {};
    
    \node (1) at (-2*\r,3*\r) {};
    \node (8) at (-2*\r,2*\r) {};
    \node (5) at (2*\r,3*\r) {};
    \node (12) at (2*\r,2*\r) {};
    \node (6) at (2*\r,1*\r) {};
    \node (13) at (3*\r,3*\r) {};
    \node (7) at (4*\r,3*\r) {};
 
\end{scope}

\begin{scope}
    \path [-] (14) edge node {} (2);
    \path [-] (14) edge node {} (3);        
    \path [-] (14) edge node {} (4);
    \path [-] (2) edge node {} (9);
    \path [-] (3) edge node {} (10);    
    \path [-] (4) edge node {} (11);
    
    \path [-] (1) edge node {} (14);
	\path [-] (1) edge node {} (8);    
    
    \path [-] (14) edge node {} (5);        
    \path [-] (12) edge node {} (5);
    \path [-] (12) edge node {} (6);
    \path [-] (5) edge node {} (13);    
    \path [-] (7) edge node {} (13);
    
\end{scope}

\end{tikzpicture}}
\caption{Tree whose smallest known connected supergraph with purely imaginary roots has order $1509397$.}\label{fig:treewithlargeclique}
\end{figure}

Although this tree is a subgraph of a graph with purely imaginary independence roots, we have yet to see any trees with purely imaginary independence roots. In fact, every graph mentioned above with purely imaginary independence roots has universal vertices and therefore many cycles. It is tempting to conjecture that every graph with purely imaginary roots must have a universal vertex and therefore that no trees have purely imaginary independence roots. However, for every graph with purely imaginary independence roots a corresponding tree can be constructed that also has these purely imaginary independence roots. This is done using a fascinating result due to Bencs \cite{Bencs2018}.

\begin{proposition}[\cite{Bencs2018}]
If $G$ is a connected graph, then there exists a tree $T$ and a sequence of induced subgraphs of $G$, $G_1,G_2,\ldots, G_k$, such that 
$$I(T,x)=I(G,x)I(G_1,x)I(G_2,x)\cdots I(G_k,x). $$
\end{proposition}
Note that the proof is constructive, but the construction is beyond the scope of this note so we direct (and encourage) the interested reader to see \cite{Bencs2018}. We do note that $T$ can be much larger than $G$. This proposition and Theorem~\ref{thm:infmanyiovernroots} also yield the following corollary.

\begin{corollary}
For all $n\in\mathbb{Z}\setminus\{0\}$, there are infinitely many trees with independence roots at $\frac{i}{n}$ and $-\frac{i}{n}$.  
\end{corollary}

We conclude by asking some open questions.
%
%
\begin{figure}[htb]
\centering
\def \n {8}
\def \r {1}
\def \radius {3cm}
\def \margin {8} 
\scalebox{0.5}{
\begin{tikzpicture}
\begin{scope}[every node/.style={circle,thick,draw,fill}]
    \node (61) at (-6*\r,4.3*\r) {};
    \node (62) at (-6.5*\r,0.9*\r) {};
    \node (63) at (-6*\r,-4.3*\r) {};
    \node (64) at (6*\r,4.3*\r) {};
    \node (65) at (6.5*\r,0.9*\r) {};
    \node (66) at (6*\r,-4.3*\r) {};
 
\end{scope}

\foreach \s in {1,...,\n}
{
  \node[circle,thick,draw,fill] (\s) at ({360/\n * (\s - 1)}:\radius) {};
}

\begin{scope}
    \path [-] (1) edge node {} (2);
    \path [-] (1) edge node {} (3);        
    \path [-] (1) edge node {} (4);
    \path [-] (1) edge node {} (5);
    \path [-] (1) edge node {} (6);    
    \path [-] (1) edge node {} (7);
    \path [-] (1) edge node {} (8);
    
    \path [-] (2) edge node {} (3);
    \path [-] (2) edge node {} (4);
    \path [-] (2) edge node {} (5);
    \path [-] (2) edge node {} (6);
    \path [-] (2) edge node {} (7);
    \path [-] (2) edge node {} (8);
    
    \path [-] (3) edge node {} (4);
    \path [-] (3) edge node {} (5);
    \path [-] (3) edge node {} (6);
    \path [-] (3) edge node {} (7);
    \path [-] (3) edge node {} (8);
    
    \path [-] (4) edge node {} (5);
    \path [-] (4) edge node {} (6);
    \path [-] (4) edge node {} (7);
    \path [-] (4) edge node {} (8);
    
    \path [-] (5) edge node {} (6);
    \path [-] (5) edge node {} (7);
    \path [-] (5) edge node {} (8);
   
    \path [-] (6) edge node {} (7);
    \path [-] (6) edge node {} (8);
    
    \path [-] (7) edge node {} (8);
    
    \path [-] (1) edge node {} (61);
    \path [-] (1) edge node {} (62);
    \path [-] (1) edge node {} (63);
    \path [-] (1) edge node {} (64);
    \path [-] (1) edge node {} (65);
    \path [-] (1) edge node {} (66);
    
    \path [-] (2) edge node {} (61);
    \path [-] (2) edge node {} (62);
    \path [-] (2) edge node {} (63);
    \path [-] (2) edge node {} (64);
    \path [-] (2) edge node {} (65);
    \path [-] (2) edge node {} (66);
    
    \path [-] (3) edge node {} (61);
    \path [-] (3) edge node {} (62);
    \path [-] (3) edge node {} (63);
    \path [-] (3) edge node {} (64);
    \path [-] (3) edge node {} (65);
    \path [-] (3) edge node {} (66);
    
    \path [-] (4) edge node {} (61);
    \path [-] (4) edge node {} (62);
    \path [-] (4) edge node {} (63);
    \path [-] (4) edge node {} (64);
    \path [-] (4) edge node {} (65);
    \path [-] (4) edge node {} (66);
    
    \path [-] (5) edge node {} (61);
    \path [-] (5) edge node {} (62);
    \path [-] (5) edge node {} (63);
    \path [-] (5) edge node {} (64);
    \path [-] (5) edge node {} (65);
    \path [-] (5) edge node {} (66);
    
    \path [-] (6) edge node {} (61);
    \path [-] (6) edge node {} (62);
    \path [-] (6) edge node {} (63);
    \path [-] (6) edge node {} (64);
    \path [-] (6) edge node {} (65);
    \path [-] (6) edge node {} (66);
    
    \path [-] (7) edge node {} (61);
    \path [-] (7) edge node {} (62);
    \path [-] (7) edge node {} (63);
    \path [-] (7) edge node {} (64);
    \path [-] (7) edge node {} (65);
    \path [-] (7) edge node {} (66);
    
    \path [-] (8) edge node {} (61);
    \path [-] (8) edge node {} (62);
    \path [-] (8) edge node {} (63);
    \path [-] (8) edge node {} (64);
    \path [-] (8) edge node {} (65);
    \path [-] (8) edge node {} (66);
  
\end{scope}

\end{tikzpicture}}
\caption{The graph of smallest known order with purely imaginary independence roots, $\overline{K_{6}}+K_{8}$.}\label{fig:smallest}
\end{figure}
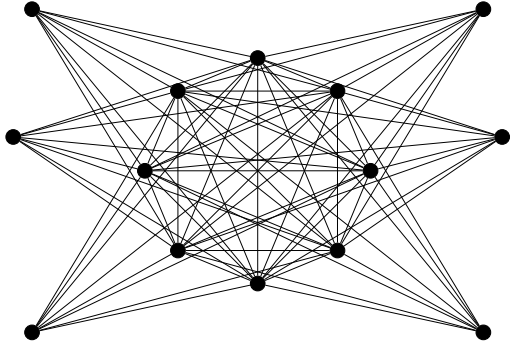

\begin{question}\label{ques:smallest}
What is the graph of smallest order with purely imaginary independence roots?
\end{question}

We have computationally verified that all graphs $G$ with $|V(G)|\le 9$ and all trees $T$ with $|V(T)|\le 20$ have no purely imaginary roots (to avoid rounding errors, our algorithm found the $\gcd$ of $I_{even}(G,x)$ and $I_{odd}(G,x)$ and then solved for any negative real roots). On the other hand, the graph of order $14$ in Figure~\ref{fig:smallest} has purely imaginary independence roots, so the answer to Question~\ref{ques:smallest} lies in the set $\{10,11,12,13,14\}$.

\begin{question}
Are there graphs with purely imaginary independence roots at $bi$ where $b$ is \textit{irrational}?
\end{question}

From Proposition~\ref{prop:rationalpureimroots} we know that there are no irrational purely imaginary independence roots of the form $i\sqrt{b}$ with $b\in\mathbb{Q}$ and $|b|>1$. The problem is open though for all other irrational purely imaginary numbers.


\subsection*{Acknowledgements}
The authors acknowledge and thank Jun Ge for catching a mistake in an earlier preprint of this note. Research of J.I.\ Brown is partially supported by grant RGPIN-2018-05227 from Natural Sciences and Engineering Research Council of Canada (NSERC).
%
%
%
%
%

\bibliographystyle{plain}
\bibliography{PureImRts}

\end{document}